\documentclass[twocolumn]{autart}    % Enable this line and disable the 
\usepackage{amsmath}
\usepackage{amssymb}
\usepackage{ntheorem}
\usepackage[natbibapa]{apacite}
\usepackage{natbib}
\usepackage{graphicx}          % Include this line if your 
\newtheorem{problem}{Problem}
\newtheorem{theorem}{Theorem}%[section]
\newtheorem{assumption}{Assumption}%[section]
\newtheorem{lemma}{Lemma}%[section]
\newtheorem{remark}{Remark}%[section]
\newtheorem{definition}{Definition}%[section]
\newtheorem*{proof}{Proof.}%[section]

\begin{document}

\begin{frontmatter}

\title{An Approach to Mismatched Disturbance Rejection Control for Uncontrollable Systems\thanksref{footnoteinfo}} % Title, preferably not more 

\thanks[footnoteinfo]{This work was supported by the Foundation for Innovative Research Groups of the National Natural Science Foundation of China (61821004), Major Basic Research of Natural Science Foundation of Shandong Province (ZR2021ZD14), High-level Talent Team Project of Qingdao West Coast New Area (RCTD-JC-2019-05), Key Research and Development Program of Shandong Province (2020CXGC01208), Science and Technology Project of Qingdao West Coast New Area (2019-32, 2020-20, 2020-1-4), National Science Foundation of China under Grant 62103241 and Shandong Provincial Natural Science Foundation under Grant ZR2021QF107, National Science and Technology Major Project (2017-V-0010-0060 and 2017-V-0013-0065), National Natural Science Foundation of China (51506176). Corresponding author Huanshui Zhang. 
}

\author[Paestum]{Shichao Lv}\ead{lv\_sc2020@163.com},    % Add the 
\author[Paestum]{Hongdan Li}\ead{lhd200908@163.com},               % e-mail address 
\author[Baiae]{Kai Peng}\ead{pengkai@nwpu.edu.cn},  % (ead) as shown
\author[Paestum]{Huanshui Zhang}\ead{hszhang@sdu.edu.cn}  % (ead) as shown

\address[Paestum]{College of Electrical Engineering and Automation, Shandong University of Science and Technology, Qingdao, Shandong, P.R.China 266590.}  % Please supply                                                         % full addresses
\address[Baiae]{School of Power and Energy, Northwestern Polytechnical University, Xi'an, Shaanxi, P.R.China 710072.}        % here.

                             % keyword list or with the 
                                          % help of the Automatica 
                                          % keyword wizard

\begin{abstract}                          % Abstract of not more than 200 words.
This study focuses on the problem of optimal mismatched disturbance rejection control for uncontrollable linear discrete-time systems. In contrast to previous studies, by introducing a quadratic performance index such that the regulated state can track a reference trajectory and minimize the effects of disturbances, mismatched disturbance rejection control is transformed into a linear quadratic tracking problem. The necessary and sufficient conditions for the solvability of this problem over a finite horizon and a disturbance rejection controller are derived by solving a forward-backward difference equation. In the case of an infinite horizon, a sufficient condition for the stabilization of the system is obtained under the detectable condition. This paper details our novel approach to disturbance rejection. Four examples are provided to demonstrate the effectiveness of the proposed method.
\end{abstract}

\end{frontmatter}

\section{Introduction}
\label{sec:introduction}
With the growing interest in high-precision control, the implementation of disturbance rejection techniques is generally required in controller design. Therefore, disturbance rejection is a fundamental issue in automatic control. The disturbances in the system are classified into matched and mismatched disturbances according to their relationship with the control input. Different disturbance rejection controllers handle disturbances using different schemes. Most studies have focused on matched disturbance rejection through active disturbance rejection control (ADRC) \citep{GUO20132911,4796887,5572931,ZHAO2014882}, disturbance observer-based control \citep{chen2015disturbance,li2014disturbance}, and sliding mode control (SMC) \citep{shtessel2014sliding,young1999control}.

Comparatively, the rejection of mismatched disturbances is more challenging. Mismatched disturbances extensively exist in the real world. Many practical systems such as permanent magnet synchronous motors, roll autopilots for missiles, and flight control systems are affected by mismatched disturbances \citep{chen2003nonlinear,chwa2004compensation,mohamed2007design}. In contrast to matched disturbances, these disturbances act on a system through a different channel than the control input, or the effects of these disturbances cannot be equivalently transformed into input channels. As a result, regardless of the type of control scheme employed, eliminating the influence of mismatched disturbances on the system state may be impossible \citep{isidori1985nonlinear}. Therefore, a practical approach is to eliminate the effects of mismatched disturbances from certain variables of interest representing the regulated state.

There are several methods for handling mismatched disturbances \citep{castillo,chen2016adrc,ginoya2013sliding,2012Generalized,9339876,yang2012nonlinear,6129407}. For a nonlinear system with mismatched disturbances, a novel SMC scheme based on a generalized disturbance observer was presented in \citep{ginoya2013sliding,6129407}. This scheme can reject mismatched disturbances in outputs in the steady state. In contrast to \citep{ginoya2013sliding,6129407}, \citep{yang2012nonlinear} treated mismatched disturbances in a multi-input multi-output system with arbitrary disturbance relative degrees. In a linear system with mismatched disturbances, a generalized extended state observer-based control (GESOBC) method was proposed for linear controllable systems to eliminate mismatched disturbances in the controlled output in the steady state \citep{2012Generalized}. Similar to \citep{2012Generalized}, \citep{castillo} weakened the restriction of disturbances and improved the disturbance rejection effect by introducing high-order derivatives of disturbances.
However, most previous works have focused on controllable systems and these previously proposed methods are not applicable to uncontrollable linear systems.
Additionally, the balance of disturbance rejection and control input energy costs was not considered in the studies mentioned above. In other words, they did not use a minimum cost control input for disturbance rejection control.

In this paper, we discuss mismatched disturbance rejection control for linear discrete-time systems.
The core of this problem is the design of a controller to weaken or eliminate the effects of disturbances on the regulated state.
To this end, we introduce a quadratic performance index such that the regulated state can track a reference trajectory and the effects
of disturbances can be minimized. From this perspective, mismatched disturbance rejection control is transformed into a linear quadratic tracking (LQT) problem. It should be noted that this problem differs from the standard LQT problem based on the presence of mismatched disturbances. To the best of our knowledge, this is a novel method for solving mismatched disturbance rejection control in terms of LQT.

The main contributions of this study can be summarized as follows. For a finite horizon, the necessary and sufficient condition for the existence of a disturbance rejection controller is derived in terms of the Riccati difference equation. Additionally, an analytical expression of the disturbance rejection controller is obtained. By using the decoupling technique, we provide a solution for the forward-backward difference equations (FBDEs) obtained by applying the maximum principle. For an infinite horizon, we provide the sufficient condition of stabilization based on the general algebraic Riccati equation (GARE) with a pseudo-inverse matrix.
Furthermore, the proposed method is extended to receding-horizon control, which can handle known disturbances in real time.
In contrast to \citep{castillo,gandhi2020hybrid,2012Generalized}, where results were obtained with the requirement of uncontrollability, we derived stabilization results under the detectability condition alone. However, the results are still applicable to some uncontrollable cases, which is demonstrated through a numerical example. Furthermore, the effectiveness and feasibility of the proposed method are illustrated through a simulation comparison between the proposed method and proportional-integral-derivative (PID) control in an aero-engine system.

The remainder of this paper is organized as follows. In Section \ref{sec2}, the problem of linear quadratic mismatched disturbance rejection control for discrete-time systems is introduced. Controller design and analysis are presented in Section \ref{sec3}. In Section \ref{sec4}, examples are provided to demonstrate the effectiveness of the proposed method. Our conclusions are summarized in the final section.

Notation: $\mathbb{R}^{n}$ represents the $n$-dimensional Euclidean space; the superscripts $'$, $^{-1}$, $^{\dagger}$, and $\Vert \cdot \Vert$ represent the transpose, inverse, pseudo-inverse, and  2-norm of a matrix, respectively; $I$ denotes the unit matrix; $O$ denotes the zero matrix; a symmetric matrix $M > 0$ ($\ge$ 0) is positive definite (positive semi-definite); and $\rho(\cdot)$ denotes a matrix eigenvalue.

%\begin{figure}
%\begin{center}
%\includegraphics[height=4cm]{jcaesar.eps}    % The printed column  
%\caption{Gaius Julius Caesar, 100--44 B.C.}  % width is 8.4 cm.
%\label{fig1}                                 % Size the figures 
%\end{center}                                 % accordingly.
%\end{figure}

% OR

%\begin{figure}
%\begin{center}
%\epsfig{file=jcaesar,width=7cm}
%\caption{Gaius Julius Caesar, 100--44 B.C.}
%\label{fig1}
%\end{center}
%\end{figure}

\section{Problem Statement}\label{sec2}
%\normalsize
Consider the following linear system with a mismatched disturbance:
\begin{align}\label{f2.1}
	      {x_{k+1}=Ax_{k}+Bu_{k}+Ed_{k}},
	      %,
\end{align}		
where $x_{k}\in\mathbb{R}^{n}, u_{k}\in\mathbb{R}^{m}$, and $d_{k}\in\mathbb{R}^{m}$ denote the state, control input, and disturbance, respectively. $A\in\mathbb{R}^{n\times n}, B, E\in\mathbb{R}^{n\times m}$ are coefficient matrices.
\begin{remark}
In (\ref{f2.1}), $d_k$ denotes a known disturbance. Known disturbances include measurable disturbances \citep{9050185,GUO201941} and disturbances for which disturbance models are available \citep{yang1994disturbance,10.1115/1.2896180}. The case in which a disturbance is only measurable at the current time can be solved by extending receding-horizon control, as detailed at the end of Section \ref{sec3}.
\end{remark}
\begin{remark}
A mismatched disturbance implies that the following matching conditions cannot be satisfied. The matching condition implies that $B=E$ (or, more precisely, $B\Gamma=E$ for some $\Gamma$) \citep{chen2015disturbance}, which means that disturbances always affect the system through the same channel as the control input or the effects of disturbances can be equivalently transformed into input channels.
\end{remark}
The mismatched disturbance rejection problem for the linear discrete-time system in (\ref{f2.1}) is transformed into an LQT problem such that the regulated state of the system optimally tracks the reference based on compensation for the disturbance.

Here, $c_ox_{k}\in\mathbb{R}^{l}$ represents the regulated state and ${c}_or$ ($r\in\mathbb{R}^{n}$) denotes the reference for the desired regulated state. Therefore, the following function can be obtained:
\begin{align}\label{f6.1}
({c}_ox_{k}-{c}_or)'({c}_ox_{k}-{c}_or)=(x_{k}-r)'{{c}_o}'{c}_o(x_{k}-r),
\end{align}
which expresses the difference between the regulated and reference states.
We define the following cost function:
\begin{align}\label{f2.3}
 J_{N}=&\sum^{N}_{k=0}[(x_{k}-r)'Q(x_{k}-r)\notag\\
&+(Bu_{k}+Ed_{k})'R(Bu_{k}+Ed_{k})]\notag\\
&+(x_{N+1}-r)'P_{N+1}(x_{N+1}-r),
\end{align}
where $Q={{c}_o}'{c}_o$ is consistent with (\ref{f6.1}) and $Q, R, P_{N+1}\in\mathbb{R}^{n\times n}$ are semi-positive definite.
\begin{remark}
The cost function is divided into two parts. The first part $(x_{k}-r)'Q(x_{k}-r)$ represents the error between the regulated and reference states. This ensures that the regulated state tracks the reference. The second part $(Bu_{k}+Ed_{k})'R(Bu_{k}+Ed_{k})$ is the sum of the control input and disturbance, where the goal is to compensate for the disturbance based on the control input.
\end{remark}
\begin{remark}
The proposed cost function can cover both matching and mismatching cases. The invertible matrix $B$ is the simplest case. Clearly, $Bu_{k}+Ed_{k}=0$ has a solution. When $B$ is irreversible, it is necessary to discuss whether a solution to $Bu_{k}+Ed_{k}=0$ exists. When $B\Gamma=E$ for some $\Gamma$ is satisfied, then a solution for $Bu_{k}+Ed_{k}=0$ exists (ADRC is suitable for this type of scenario). Additionally, it is more difficult to handle the case where $\Gamma$ does not satisfy $B\Gamma=E$, which is the main focus of this study.
\end{remark}
\begin{problem}\label{p1}
We must find a control $u_{k}$ such that the effects of the disturbance $d_{k}$ are minimized, that is, $Bu_{k}+Ed_{k}$ is minimized, and the state $x_{k}$ tracks the reference trajectory $r$.
\end{problem}

\section{Controller Design}\label{sec3}
The primary goal of controller design is to develop a stabilizing disturbance rejection controller to force the regulated state of a system to follow the reference. For the linear system in (\ref{f2.1}) with the cost function in (\ref{f2.3}), the LQT problem can be solved using the FBDEs derived from the maximum principle to obtain a disturbance rejection controller.
\subsection{Optimization}

First, the following lemma is provided:

\begin{lemma}\citep{zhang2018optimal}\label{l1}
Problem \ref{p1} is uniquely solvable if and only if the following FBDEs have a unique solution:
\begin{eqnarray}
\left\{
\begin{array}{lll}
0=B'RBu_{k}+B'\lambda_{k}+B'REd_{k}\\
\lambda_{k-1}=Q(x_{k}-r)+A'\lambda_{k}\\
x_{k+1}=Ax_{k}+Bu_{k}+Ed_{k}\\
\lambda_{N}=P_{N+1}(x_{N+1}-r).
\end{array}
\right.\label{f2.10}
\end{eqnarray}
\end{lemma}
\begin{remark}\label{r5101}
From Lemma \ref{l1}, we know that the key to solving Problem \ref{p1} is the solvability of the FBDEs in (\ref{f2.10}).
\end{remark}
We define the following Riccati equation:
\begin{align}
P_{k}=Q+A'P_{k+1}A-M'_{k}\Upsilon^{-1}_{k}M_{k},\label{f2.8}
\end{align}
where
\begin{align}
\Upsilon_{k}&=B'(R+P_{k+1})B,\label{f2.5}\\
M_{k}&=B'P_{k+1}A,\label{f2.6}
\end{align}
with a terminal value of $P_{N+1}$.
\begin{theorem}\label{t1}
  A unique controller $u_{k}$ that minimizes the effect of the disturbance exists if and only if $\Upsilon_{k}>0, k=0,1,\ldots, N$. In this case, the optimal controller is given by
\begin{align}
u_{k}=-\Upsilon^{-1}_{k}M_{k}x_{k}-\Upsilon^{-1}_{k}h_{k},
 \label{f2.4}
\end{align}
where $h_{k}$ satisfies the following backward equation:
\begin{equation}
\left\{
\begin{aligned}
h_{k}=B'&(R+P_{k+1})Ed_{k}+B'f_{k+1},\\
f_{k}=A'&P_{k+1}Ed_{k}+A'f_{k+1}-M'_{k}\Upsilon^{-1}_{k}h_{k}-Qr,\\
f_{N+1}=&-P_{N+1}r.
\end{aligned}
\right.
 \end{equation}

%\begin{eqnarray}
%\left\{
%\begin{aliged}{lll}
%h_{k}=&B'(R+P_{k+1})Ed_{k}+B'f_{k+1},\\
%f_{k}=&A'P_{k+1}Ed_{k}+A'f_{k+1}\\
%&-M'_{k}\Upsilon^{-1}_{k}h_{k}-Qr,\\
%f_{N+1}=&-P_{N+1}r.
%\end{aliged}
%\right.\label{f2.7}%\label{f2.9}
%\end{eqnarray}
The optimal cost function is given as
\begin{align}
J_{N}=&x'_{0}P_{0}x_{0}+2x'_{0}f_{0}+r'P_{N+1}r\notag\\
&+\sum^{N}_{k=0}[r'Qr+d'_{k}E'(R+P_{k+1})Ed_{k}\label{f2.020}\\
&+2d'_{k}E'f_{k+1}-h'_{k}\Upsilon^{-1}_{k}h_{k}]\notag.
\end{align}
Additionally, the relationship between the state $x_{k}$ and co-state $\lambda_{k}$ is defined as follows:
\begin{align}
\lambda_{k-1}=P_{k}x_{k}+f_{k}. \label{5103}
\end{align}
\end{theorem}
\begin{proof}
\textbf{Necessity.}\ \  When Problem \ref{p1} is uniquely solvable, we prove that $\Upsilon_{k}>0, k=0,1,\ldots, N$.

To this end, we define
\begin{align}
\bar{J}_{k}=&\sum^{N}_{i=k}[(x_{i}-r)'Q(x_{i}-r)\notag\\
&+(Bu_{i}+Ed_{i})'R(Bu_{i}+Ed_{i})]\notag\\
&+(x_{N+1}-r)'P_{N+1}(x_{N+1}-r). \label{5104}
\end{align}

For $k=N$, according to the semi-positive definiteness of $Q, R$, and $P_{N+1}$, we obtain
\begin{align}
\bar{J}_{N}=&(x_{N}-r)'Q(x_{N}-r)\notag\\
&+(x_{N+1}-r)'P_{N+1}(x_{N+1}-r)\notag\\
&+(Bu_{N}+Ed_{N})'R(Bu_{N}+Ed_{N})\notag\\
=&u_{N}'\Upsilon_{N}u_{N}+2u_{N}'B'[(R+P_{N+1})Ed_{N}\notag\\
&+P_{N+1}(Ax_{N}-r)]+2(Ax_{N}-r)'P_{N+1}Ed_{N}\notag\\
&+d_{N}'E'(R+P_{N+1})Ed_{N}+(x_{N}-r)'Q(x_{N}-r)\notag\\
&+(Ax_{N}-r)'P_{N+1}(Ax_{N}-r)\notag\\
 \geq&0. \label{5105}
\end{align}

When $x_{N}=0$, based on the unique existence of a solution to Problem \ref{p1}, we find that $\bar{J}_{N}>0$ for any nonzero $u_{N}$. Therefore, $\Upsilon_{N}>0$. From (\ref{f2.10}), we obtain
\begin{align}
0=&B'RBu_{N}+B'\lambda_{N}+B'REd_{N}\notag\\
=&B'RBu_{N}+B'REd_{N}\notag\\
&+B'P_{N+1}[Ax_{N}+Bu_{N}+Ed_{N}-r]\notag\\
=&B'(R+P_{N+1})Bu_{N}+B'P_{N+1}Ax_{N}\notag\\
&+B'(R+P_{N+1})Ed_{N}-B'P_{N+1}r.\label{f2.11}
\end{align}
Therefore,
\begin{align}
u_{N}=-\Upsilon^{-1}_{N}M_{N}x_{N}-\Upsilon^{-1}_{N}h_{N}.
 \label{f2.12}
\end{align}
Additionally,
\begin{align}
\lambda_{N-1}=&Q(x_{N}-r)+A'P_{N+1}[Ax_{N}+Bu_{N}+Ed_{N}-r]\notag\\
=&[Q+A'P_{N+1}A-M'_{N}\Upsilon^{-1}_{N}M_{N}]x_{N}-Qr\notag\\
&+A'P_{N+1}Ed_{N}-M'_{N}\Upsilon^{-1}_{N}h_{N}-A'P_{N+1}r\notag\\
=&P_{N}x_{N}+f_{N}. \label{f2.13}
\end{align}

We use a mathematical inductive method to prove the necessity in the following analysis. Therefore, we assume that for $n+1\leq k \leq N$, the following holds:
\begin{align}
& 1)\ \  \Upsilon_{k}>0; \label{5303}\\
& 2)\ \  u_{k}=-\Upsilon^{-1}_{k}M_{k}x_{k}-\Upsilon^{-1}_{k}h_{k};\label{5304}\\
& 3)\ \  \lambda_{k-1}=P_{k}x_{k}+f_{k}. \label{5305}
\end{align}

We define
\begin{align}
L_{k}=x'_{k}P_{k}x_{k}+2x'_{k}f_{k}. \label{f2.17}
\end{align}
Then,
\begin{align}
L_{k+1}=&x'_{k+1}P_{k+1}x_{k+1}+2x'_{k+1}f_{k+1}\notag\\
=&x'_{k}A'P_{k+1}Ax_{k}+x'_{k}A'P_{k+1}Bu_{k}\notag\\
&+x'_{k}A'P_{k+1}Ed_{k}+u'_{k}B'P_{k+1}Ax_{k}\notag\\
&+u'_{k}B'P_{k+1}Bu_{k}+u'_{k}B'P_{k+1}Ed_{k}\notag\\
&+d'_{k}E'P_{k+1}Ax_{k}+d'_{k}E'P_{k+1}Bu_{k}+d'_{k}E'f_{k+1}\notag\\
&+d'_{k}E'P_{k+1}Ed_{k}+x'_{k}A'f_{k+1}+u'_{k}B'f_{k+1}\notag\\
&+f'_{k+1}Ax_{k}+f'_{k+1}Bu_{k}
+f'_{k+1}Ed_{k}. \label{f2.18}
\end{align}
Therefore, for $n+1\leq k \leq N$, based on the assumption above, we have
\begin{align}
&L_{k}-L_{k+1}\notag\\
=&x'_{k}(P_{k}-A'P_{k+1}A)x_{k}+x'_{k}[f_{k}-A'P_{k+1}Ed_{k}-A'f_{k+1}]\notag\\
&+[f'_{k}-d'_{k}E'P_{k+1}A-f'_{k+1}A]x_{k}
-x'_{k}M'_{k}u_{k}\notag\\
&-u'_{k}M_{k}x_{k}-u'_{k}(\Upsilon_{k}-B'RB)u_{k}-d'_{k}E'P_{k+1}Ed_{k}\notag\\
&-u'_{k}[B'P_{k+1}Ed_{k}+B'f_{k+1}]-d'_{k}E'f_{k+1}-f'_{k+1}Ed_{k}\notag\\
&-[d'_{k}E'P_{k+1}B+f'_{k+1}B]u_{k}\notag\\
=&x'_{k}(Q-M'_{k}\Upsilon^{-1}_{k}M_{k})x_{k}
-x'_{k}[M'_{k}\Upsilon^{-1}_{k}h_{k}+Qr]\notag\\
&-[M'_{k}\Upsilon^{-1}_{k}h_{k}+Qr]'x_{k}
-x'_{k}M'_{k}u_{k}-f'_{k+1}Ed_{k}\notag\\
&-u'_{k}M_{k}x_{k}-u'_{k}(\Upsilon_{k}-B'RB)u_{k}-u'_{k}[h_{k}-B'REd_{k}]\notag\\
&-[h_{k}-B'REd_{k}]'u_{k}-d'_{k}E'P_{k+1}Ed_{k}-d'_{k}E'f_{k+1}\notag\\
=&[(x_{k}-r)'Q(x_{k}-r)+(Bu_{k}+Ed_{k})'R(Bu_{k}+Ed_{k})]\notag\\
&-r'Qr-d'_{k}E'REd_{k}-(u_{k}+\Upsilon^{-1}_{k}M_{k}x_{k}+\Upsilon^{-1}_{k}h_{k})'\Upsilon_{k}\notag\\
&\times(u_{k}+\Upsilon^{-1}_{k}M_{k}x_{k}+\Upsilon^{-1}_{k}h_{k})+h'_{k}\Upsilon^{-1}_{k}h_{k}\notag\\
&-d'_{k}E'P_{k+1}Ed_{k}-d'_{k}E'f_{k+1}-f'_{k+1}Ed_{k}.
\label{f2.19}
\end{align}
By summing (\ref{f2.19}) from $k=n+1$ to $k=N$, we obtain
\begin{align}
\bar{J}_{n+1}=&x'_{n+1}P_{n+1}x_{n+1}+2x'_{n+1}f_{n+1}+r'P_{N+1}r\notag\\
&+\sum^{N}_{k=0}[r'Qr+d'_{k}E'(R+P_{k+1})Ed_{k}\notag\\
&+2d'_{k}E'f_{k+1}-h'_{k}\Upsilon^{-1}_{k}h_{k}].\label{f2.20}
\end{align}

Accordingly, when $x_{n}=0$, we have
\begin{align}
\bar{J}_{n}=&u_{n}'\Upsilon_{n}u_{n}+2u_{n}'B'[(R+P_{n+1})Ed_{n}+f_{n+1}]\notag\\
&+d_{n}'E'[(R+P_{n+1})Ed_{n}+f_{n+1}]\notag\\
&+\sum^{N}_{k=n+1}[r'Qr+d'_{k}E'(R+P_{k+1})Ed_{k}\notag\\
&+2d'_{k}E'f_{k+1}-h'_{k}\Upsilon^{-1}_{k}h_{k}]+r'(Q+P_{N+1})r.\label{f2.21}
\end{align}

Therefore, there is a unique $u_{n}$ that minimizes $\bar{J}_{n}$ only if $\Upsilon_{n}$ is invertible. In this case, according to (\ref{5303}), (\ref{5304}), and (\ref{5305}), by following the processes of (\ref{f2.11}) and (\ref{f2.13}), the optimal controller $u_{n}$ is defined in (\ref{f2.4}) and $\lambda_{n-1}$ is equivalent to (\ref{5103}) with $k=n$. By induction, the proof of necessity for all $0\leq k \leq N$ is complete.

\textbf{Sufficiency.}\ \ For (\ref{f2.19}), performing the summation from $k=0$ to $N$ yields 
\begin{align}
J_{N}=&x'_{0}P_{0}x_{0}+2x'_{0}f_{0}+r'P_{N+1}r\notag\\
&+\sum^{N}_{k=0}[r'Qr+d'_{k}E'REd_{k}-h'_{k}\Upsilon^{-1}_{k}h_{k}\notag\\
&+(u_{k}+\Upsilon^{-1}_{k}M_{k}x_{k}+\Upsilon^{-1}_{k}h_{k})'\Upsilon_{k}\notag\\
&\times(u_{k}+\Upsilon^{-1}_{k}M_{k}x_{k}+\Upsilon^{-1}_{k}h_{k})\notag\\
&+d'_{k}E'P_{k+1}Ed_{k}+d'_{k}E'f_{k+1}+f'_{k+1}Ed_{k}].
\label{f2.22}
\end{align}

Based on the positive definiteness of $\Upsilon_{k}$ for $0\leq k \leq N$, it is easy to deduce the unique solvability of Problem \ref{p1}.

\end{proof}

\begin{remark}\label{rem.3}
It is noteworthy that the controller proposed in Theorem \ref{t1} can optimally eliminate disturbances in the regulated state. To the best of our knowledge, this result is novel.
\end{remark}

\begin{remark}\label{r6}
The explicit expressions for $h_{k}$ and $f_{k}$ are given as follows:
\begin{align}
h_{k}=&H_{k}d_{k}+B'\sum^{N}_{s=k+1}\bar{A}^{N-1}_{s}F_{s}d_{s}
-\mathcal{R}_{k+1}r,\label{51101} \\
f_{k}=&\sum^{N}_{s=k}\bar{A}^{N-1}_{s}F_{s}d_{s}-\mathcal{R}_{k}r, \label{51102}
\end{align}
where
\begin{align}
H_{k}=&B'(R+P_{k+1})E,\notag\\
\bar{A}^{N-1}_{s}=&\bar{A}_{s}'\ldots \bar{A}_{N-1}', \ \ \bar{A}_{k}=A-B\Upsilon^{-1}_{k}M_{k},\notag\\
F_{k}=&(\bar{A}_{k}'P_{k+1}-M'_{k}\Upsilon^{-1}_{k}B'R)E,\notag\\
\mathcal{R}_{k}=&\bar{A}_{k}'\mathcal{R}_{k+1}+Q, \ \  \mathcal{R}_{N+1}=P_{N+1}.\notag
\end{align}

\end{remark}

Next, we analyze the stability of the designed controller in (\ref{f2.4}).
\subsection{Stabilization}
First, we denote $P_{k}, M_{k}$, and $\Upsilon_{k}$ with the terminal time $N$ as $P_{k}(N), M_{k}(N)$, and $\Upsilon_{k}(N)$, respectively. Under the regular condition that
\begin{align}
\Upsilon_{k}(N)\Upsilon^{\dagger}_{k}(N)M_{k}(N)=M_{k}(N),\label{5600}
\end{align}
we introduce the general difference Riccati equation (GDRE) with
\begin{align}
P_{k}(N)=&Q+A'P_{k+1}(N)A-M'_{k}(N)\Upsilon^{\dagger}_{k}(N)M_{k}(N),\label{5501}
\end{align}
where
\begin{align}
\Upsilon_{k}(N)=&B'(R+P_{k+1}(N))B,\label{5502}\\
M_{k}(N)=&B'P_{k+1}(N)A,\label{5503}
\end{align}
with the terminal value $P_{N+1}=0$.

\begin{lemma}\label{l5}
Suppose that the GDRE in (\ref{5501})--(\ref{5503}) yields a solution. Then, Problem \ref{p1} has a solution expressed as
\begin{align}
u_{k}=-\Upsilon^{\dagger}_{k}(N)M_{k}(N)x_{k}-\Upsilon^{\dagger}_{k}(N)h_{k}(N),
 \label{5504}
\end{align}
where $h_{k}(N)$ satisfies the following backward equation:
\begin{equation}
\left\{
\begin{aligned}
h_{k}=B'&(R+P_{k+1})Ed_{k}+B'f_{k+1},\\
f_{k}=A'&P_{k+1}Ed_{k}+A'f_{k+1}-M'_{k}\Upsilon^{-1}_{k}h_{k}-Qr,\\
f_{N+1}=&-P_{N+1}r.
\end{aligned}\label{5505}
\right.
 \end{equation}
%\begin{equation}
%\left\{
%\begin{aligned}
%h_{k}(N)=&B'(R+P_{k+1}(N))Ed_{k}+B'f_{k+1}(N),\\
%f_{k}(N)=&A'P_{k+1}(N)Ed_{k}+A'f_{k+1}(N)-M'_{k}(N)\Upsilon^{\dagger}_{k}(N)h_{k}(N)-Qr,\\  
%f_{N+1}=&0.
%\end{aligned}
%\right\label{5505}
% \end{equation}
The optimal cost function is given as
\begin{align}
J_{N}=&x'_{0}P_{0}(N)x_{0}+2x'_{0}f_{0}(N)\notag\\
&+\sum^{N}_{k=0}[r'Qr+d'_{k}E'(R+P_{k+1}(N))Ed_{k}\notag\\
&+2d'_{k}E'f_{k+1}(N)-h'_{k}\Upsilon^{\dagger}_{k}(N)h_{k}(N)].\label{5506}
\end{align}
\end{lemma}

\begin{proof}
The proof of Theorem \ref{t1} can be derived similarly under the regular condition (\ref{5600}). To avoid redundancy, we have omitted this proof here.
\end{proof}

In this section, we consider the cost function over the infinite horizon as follows:
\begin{align}\label{5401}
 J=&\lim_{N\rightarrow \infty}\frac{1}{N}\sum^{N}_{k=0}[(x_{k}-r)'Q(x_{k}-r)\notag\\
&+(Bu_{k}+Ed_{k})'R(Bu_{k}+Ed_{k})].
\end{align}

We now introduce certain definitions and assumptions.
\begin{definition}\label{d562}
The system $(A,Q^{\frac{1}{2}})$
\begin{equation}
\left\{
\begin{aligned}
x_{k+1}=&Ax_{k},\\
y_{k}=&Q^{\frac{1}{2}}x_{k}
\end{aligned}
\right.\label{5603}
 \end{equation}
is detectable if for any $N\geq 0$, the following holds:
\begin{eqnarray*}
y_{k}=0, \forall 0\leq k \leq N \Rightarrow \lim_{k\rightarrow \infty} x_{k}=0.
\end{eqnarray*}
\end{definition}
\begin{assumption}\label{a51}
$d_{k}, k\geq 0$ is bounded and $\lim_{k\rightarrow \infty}d_{k}=d.$
\end{assumption}
\begin{assumption}\label{a52}
 $(A,\sqrt{Q})$ is detectable.
\end{assumption}
We define the GARE as follows:
\begin{align}
P=&Q+A'PA-M'\Upsilon^\dag M,\label{5403}
\end{align}
where
\begin{align}
M=&B'PA,\label{5404}\\
\Upsilon=&B'(R+P)B.\label{5405}
\end{align}

Now, consider the following system without disturbances:
\begin{align}
x_{k+1}=Ax_{k}+Bu_{k}. \label{5601}
\end{align}
Additionally, consider the following cost function:
\begin{align}
\bar{J}=\sum^{\infty}_{k=0}[x_{k}'Qx_{k}+u_{k}'B'RBu_{k}],\label{5604}
\end{align}
 where $Q$ and $R$ are both semi-positive definite.

Recall the problem of finding $u_{k}$ to stabilize the system (\ref{5601}) and minimize the cost function (\ref{5604}), which can be expressed as follows.
\begin{lemma}\label{l56}\citep{zhang2018optimal}
Suppose Assumption \ref{a52} holds and the system in (\ref{5601}) can be stabilized if and only if the GDRE in (\ref{5401}) converges when $N\rightarrow \infty$, that is, $\lim_{N\rightarrow \infty}P_{k}(N)=P$. Furthermore, $P$ is a solution of the GARE in (\ref{5403})–-(\ref{5405}) and $P\geq 0$. In this case, the stabilizing controller is
\begin{align}
u_{k}=-\Upsilon^{\dagger}Mx_{k}, \label{5410}
\end{align}
and the optimal cost value is
\begin{align}
\bar{J}^{\ast}=x_{0}'Px_{0}. \label{5411}
\end{align}
\end{lemma}
Based on the existence of the disturbance $d_{k}$, we consider the boundedness of $h_{k}$ in the following lemma.
\begin{lemma}\label{l5111}
Let Assumptions \ref{a51} and \ref{a52} hold true. When $N\rightarrow \infty$, the GDRE in (\ref{5401}) converges, that is, $\lim_{N\rightarrow \infty}P_{k}(N)=P$, and there is a constant $G>0$ such that $\|h_{k}(N)\|$ in (\ref{5505}) is bounded.
The explicit expressions for $h_{k}$ and $f_{k}$ are given as follows:
\begin{align}
h_{k}=&H_{k}d_{k}+B'\sum^{N}_{s=k+1}\bar{A}^{N-1}_{s}F_{s}d_{s}
-\mathcal{R}_{k+1}r,\label{51101} \\ 
f_{k}=&\sum^{N}_{s=k}\bar{A}^{N-1}_{s}F_{s}d_{s}-\mathcal{R}_{k}r, \label{51102}
\end{align}
where
\begin{align}
\bar{A}^{N-1}_{s}=&\bar{A}_{s}'\ldots \bar{A}_{N-1}'\notag.
\end{align}
\end{lemma}

Based on the above preliminaries, we present the main results in this section.
{\theorem \label{t2} \ \
Let Assumptions \ref{a51} and \ref{a52} hold true. If the GARE in (\ref{5403})-(\ref{5405}) has a semi-positive definite solution $P$, then the system in (\ref{f2.1}) is bounded and can be stabilized.
Under such conditions, the optimal stabilizing solution can be derived as
\begin{align}
u_{k}=-\Upsilon^{\dagger}Mx_{k}-\Upsilon^{\dagger}h_{k}.
 \label{5406}
\end{align}
}\\
%%
%And the optimal cost function is given by
%\begin{eqnarray}
%J&=&\big\{r-2[I-(A-B\Upsilon^{\dagger}M)]^{-1}Ed\big\}'Qr-h'\Upsilon^{\dagger}h\no\\
%&&+d'E'\big\{(R+P)+2[I-(A-B\Upsilon^{\dagger}M)']^{-1}[A'P-M'\Upsilon^{\dagger}B'(R+P)]
%\big\}Ed.\label{5408}
%\end{eqnarray}
\begin{proof} \ Suppose the GARE in (\ref{5403})-(\ref{5405}) has a solution $P\geq 0$. We will prove the bounded stabilization of the system in (\ref{f2.1}).

First, we demonstrate the boundedness of $h_{k}$. From Lemma \ref{l5}, $\rho(A-B\Upsilon^{\dagger}M)<1$.
%Hence, the invertibility of $I-(A-B\Upsilon^{\dagger}M)$ can be guaranteed.
Considering the convergence of $P_{k}(N)$, we can find that $H_{k}(N), F_{k}(N)$, and $\mathcal{R}_{k}(N)$ defined in Remark \ref{r6} are convergent. Therefore, there exist constants $\mathcal{C}_{H}, \mathcal{C}_{F},  \mathcal{C}_{\mathcal{R}}$ satisfying $H_{k}\leq \mathcal{C}_{H}, F_{k}\leq \mathcal{C}_{F}, \mathcal{R}_{k}\leq \mathcal{C}_{\mathcal{R}}$.
\end{proof}

From (\ref{51101}) and Assumption \ref{a51}, we have
\begin{align}
\|\sum^{\infty}_{s=k+1}\bar{A}^{s}F_{s}d_{s}\|\leq
\sum^{\infty}_{s=k+1}&\|\bar{A}\|^s\|F_{s}\|\|d_{s}\|
\notag\\
&\leq \mathcal{C}_{F} \bar{d}\sum^{\infty}_{s=k+1}\|\bar{A}\|^s
\leq \mathcal{C},
\end{align}
where $\rho(\bar{A})<1$ guarantees the boundedness of  $\sum^{\infty}_{s=k+1}\|\bar{A}\|^s$. Therefore, $h_{k}$ is bounded. Similarly, $f_{k}$ is bounded.

Note that when $u_{k}=-\Upsilon^{\dagger}Mx_{k}-\Upsilon^{\dagger}h_{k}$,  we obtain
\begin{align}
x_{k+1}=(A-B\Upsilon^{\dagger}M)x_{k}-B\Upsilon^{\dagger}h_{k}.\label{5408}
\end{align}
From the boundedness of $h_{k}$ and $\rho(A-B\Upsilon^{\dagger}M)<1$, we find that the system in (\ref{f2.1}) has bounded stability with $u_{k}=-\Upsilon^{\dagger}Mx_{k}-\Upsilon^{\dagger}h_{k}$.

\begin{remark}
The disturbance rejection controller presented in this paper is based on a known disturbance. In the future, we will further study the design of a disturbance rejection controller with mismatched unknown disturbances (uncertainties) based on the concepts and methods presented in this paper.
\end{remark}
\begin{remark}
The key concept of this study is to transform mismatched disturbance rejection control into an LQT problem. Accordingly, in contrast to the method in \citep{gandhi2020hybrid,2012Generalized,castillo}, where the system must be controllable, the stabilization result in Theorem \ref{t2} is obtained only under the detectable assumption.
\end{remark}

\begin{remark}\label{rem.9}
%%We will use a rolling horizon optimization strategy in the controller in Theorem \ref{t1} to handle broader disturbances, such as time-varying disturbances. The purpose of this method is to quickly eliminate the disturbance based on the known current time of the disturbance. The specific method is to make the disturbance in the optimized time domain the disturbance at the current time, calculate the control input that minimizes the performance index, and bring it into the control input at the current time.
Faced with a scenario in which a disturbance is only available at the current moment, we can consider using the receding-horizon control method to design a controller to handle such a disturbance. The specific controller design is as follows:
\begin{align}
u_{s}=-\Upsilon^{-1}_{s}M_{s}x_{s}-\Upsilon^{-1}_{s}h_{s}
 \label{f3.46}
\end{align}
for $s=k,k+1,\ldots,k+T$ ($T$ is a finite positive integer), where $h_{s}$, $\Upsilon_{s}$, and $M_{s}$ satisfy the following backward equations at time $k$:

\begin{eqnarray}
\left\{
\begin{array}{lll}
h_{s}=B'(R+P_{s+1})Ed+B'f_{s+1},\\
f_{s}=A'P_{s+1}Ed+A'f_{s+1}-M'_{s}\Upsilon^{-1}_{s}h_{s}-Qr,\\
\Upsilon_{s}=B'(R+P_{s+1})B,\\
M_{s}=B'P_{s+1}A,\\
d=d_{k},\\
f_{T+k+1}=-P_{T+k+1}r.
\end{array}
\right.\label{f3.47}%\label{f2.9}
\end{eqnarray}
\end{remark}

\section{Numerical Examples}\label{sec4}
In this section, four examples illustrating the effectiveness of the proposed controller are presented from different perspectives. The first example highlights the mismatched disturbance rejection effect of the proposed method in an uncontrolled system, and the second and third examples compare the disturbance rejection effect of GESOBC to that of the method proposed in this study for controlled systems with time-invariant and time-varying disturbances, respectively. The final example is the disturbance rejection application of the proposed method to an aero-engine model.
\subsection{Example A: Disturbance rejection for an uncontrollable system}
In the case where a system is stable but uncontrollable, the control law in (\ref{f3.46}) is verified to reject a mismatched disturbance.
Consider the system in (\ref{f2.1}) with the following parameters:
\begin{equation}
\begin{aligned}\label{4.3}
&{A}=\begin{bmatrix}0.96&0&0\cr 0&1&0.01\cr0&-0.02&0.99\end{bmatrix}, {B}=\begin{bmatrix}0\cr 0\cr 0.01\end{bmatrix},\notag\\
&{E}=\begin{bmatrix}0&0.01&0\end{bmatrix}', {c}_o=\begin{bmatrix}0&1&0\end{bmatrix}.\notag
\end{aligned}
\end{equation}
\begin{remark}
It is trivial to determine that the state $x^1$ in the above system is stable but not controllable, and we attempt to demonstrate the superiority of our proposed controller.
\end{remark}

The initial state of the system is $x_0 =\begin{bmatrix}1&1&0\end{bmatrix}'$ and the disturbance $d=3$ acts on the system from $k = 500$. The controller aims to remove the disturbance from the regulated state $x^2=c_ox$. In the proposed control law in (\ref{f3.46}), the horizon $T$ is set to $100$ and the terminal condition $P_{T+k+1}=O_{3\times3}$.
The reference $c_or$ is set to $0$ according to the goals defined above. The weight matrix is selected as $R=I_{3\times3}$. $P_k$, $f_k$, $h_k$, $M_k$, and $\Upsilon_{k}$ can be calculated according to (\ref{f3.47}) at every time instance $k$ to derive $u_{k}$. The simulation results for Example A are presented in Fig. \ref{fig_1}. %The simulation trajectories of the regulated state $x^2$ and the states $x^1$ as well as $x^3$  The trajectory of the disturbance is shown in Fig. \ref{fig_1}(d).

%\begin{figure}[htbp]
%    \begin{center}      
%        \includegraphics[width=9cm,height=6cm]{AA1.eps}
%        \caption{State $x^1$}\label{fig_2}
%    \end{center}%
%\end{figure}
%\begin{figure}[htbp]
%    \begin{center}      
%        \includegraphics[width=9cm,height=6cm]{AA2.eps}
%        \caption{State $x^2$}\label{fig_2.1}
%    \end{center}%
%\end{figure}
%\begin{figure}[htbp]
%    \begin{center}      
%        \includegraphics[width=9cm,height=6cm]{AA3.eps}
%        \caption{State $x^3$}\label{fig_2}
%    \end{center}%
%\end{figure}
%\begin{figure}[htbp]
%    \begin{center}      
%        \includegraphics[width=9cm,height=6cm]{AA4.eps}
%        \caption{Trajectory of the disturbance.}\label{fig_2.1}
%    \end{center}%
%\end{figure}

\begin{figure}[htbp]
    \begin{center}
%        \centering
        \includegraphics[width=9.0cm,height=8cm]{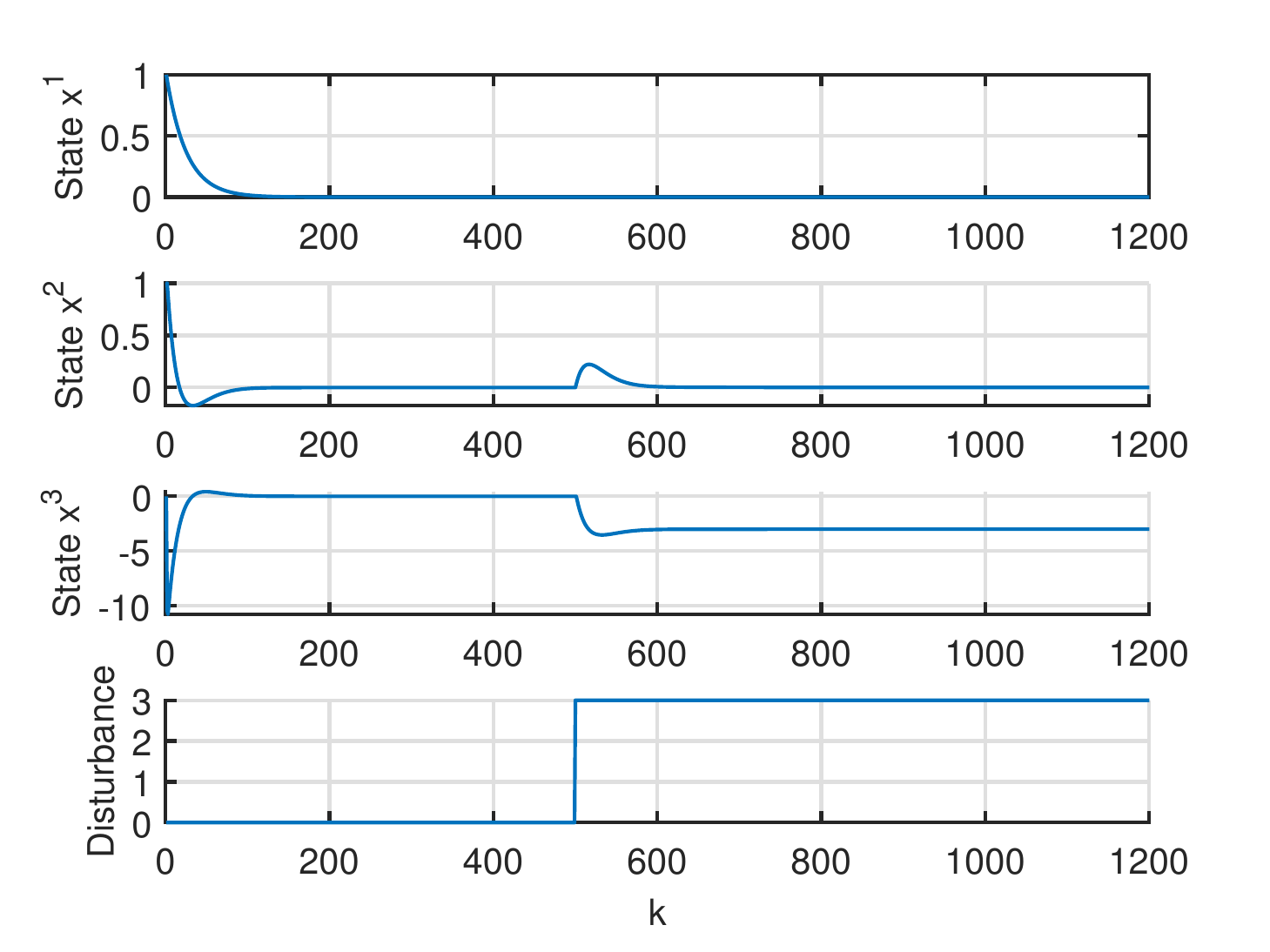}
%        \centerline{(d) Trajectory of the disturbance.}
    \end{center}
    \caption{Simulation result of Example A}\label{fig_1}
\end{figure}
In Fig. \ref{fig_1}, one can see that state $x^2$ achieves disturbance rejection quickly and stabilization is achieved for the uncontrollable state $x^1$. Therefore, the proposed method is effective for the disturbance rejection of uncontrollable systems with mismatched disturbances.

\subsection{Example B: Disturbance rejection compared to GESOBC}
This example compares GESOBC \citep{2012Generalized} to the proposed control law in (\ref{f3.46}) for constant mismatched disturbance rejection.
Consider the system in (\ref{f2.1}) with the following parameters:
\begin{equation}
\begin{aligned}\label{ff5.2}
&{A}=\begin{bmatrix}1&0.01\cr -0.02&0.99\end{bmatrix},{B}=\begin{bmatrix}0\cr 0.01\end{bmatrix},\notag\\
&{E}=\begin{bmatrix}0.01& 0\end{bmatrix}',{c}_o=\begin{bmatrix}1&0\end{bmatrix}.\notag
\end{aligned}
\end{equation}

In the proposed control law in (\ref{f3.46}), $R=I_{2\times2}$, $P_{T+k+1}=O_{2\times2}$, and the other parameters are set to be the same as in Example A.
According to \citep{2012Generalized}, the state feedback matrix $k_x$ is chosen as $[-20 -4]$ and the disturbance compensation gain of the GESOBC method is calculated as $K_d=-5$. The initial state of the system is $x_0 =\begin{bmatrix}1&0\end{bmatrix}'$ and the disturbance $d=3$ acts on the system from $k = 500$. 
Eliminating the disturbance from the regulated state $x^1={c}_ox$ is the goal of the controller. The simulation results for Example B are presented in Fig. \ref{fig_2}.
 %The simulation trajectories of the regulated state $x^1$ and the state $x^2$ are shown in Figs. \ref{fig_2}(a)-(b). The trajectory of the disturbance is shown in Fig. \ref{fig_4}.

%\begin{figure}[htbp]
%    \begin{center}      
%        \includegraphics[width=9cm,height=5.5cm]{C1.eps}
%        \caption{State $x^1$}\label{fig_2}
%    \end{center}%
%\end{figure}
%\begin{figure}[htbp]
%    \begin{center}      
%        \includegraphics[width=9cm,height=6cm]{C2.eps}
%        \caption{State $x^2$}\label{fig_2.1}
%    \end{center}%
%\end{figure}

\begin{figure}[htbp]
    \begin{center}
        \includegraphics[width=9.0cm,height=7.0cm]{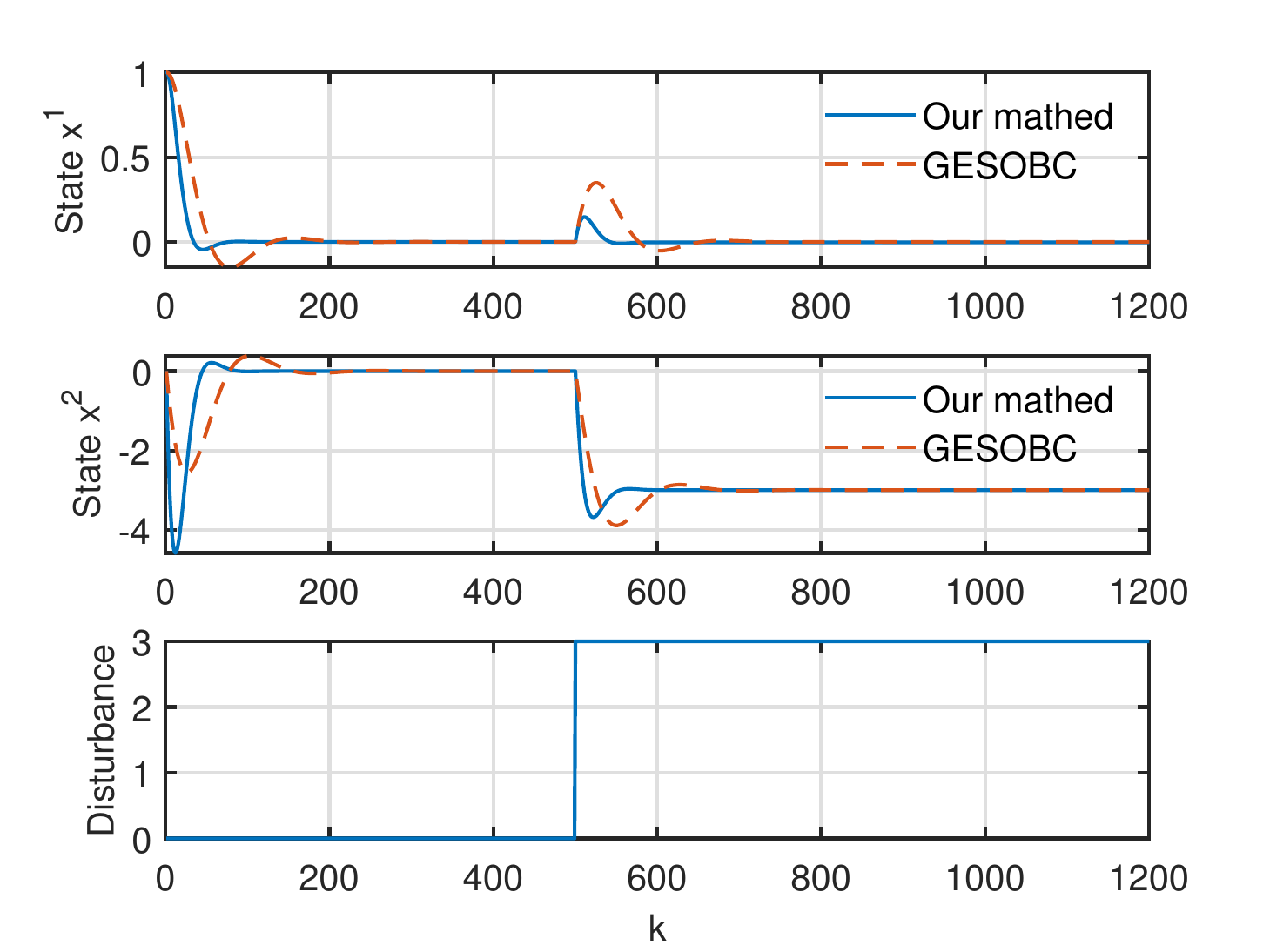}
%        \centerline{(b) State $x^2$}
    \end{center}
    \caption{Simulation result of Example B}\label{fig_2}
\end{figure}
%\begin{figure}[htbp]
%    \begin{center}      
%        \includegraphics[width=9cm,height=5.5cm]{C3.eps}
%        \caption{Disturbance of the Example B.}\label{fig_4}
%    \end{center}%
%\end{figure}

In Fig. \ref{fig_2}, one can see that the proposed method is more effective for disturbance rejection than GESOBC because the proposed method rapidly eliminates the disturbance in the regulated state $x^1$.

\subsection{Example C: Disturbance rejection compared to GESOBC for a time-varying disturbance}
Having compared the disturbance elimination effects of the two methods for a time-invariant mismatched disturbance, we now compare the disturbance rejection effects of the two methods for a time-varying mismatched disturbance.

Compared to Example B, we change the disturbance to $d_k=sin(k-500)/50$ and change $c_o=\begin{bmatrix}10&0\end{bmatrix}$, while the control objective and other parameters are consistent with Example B. The disturbance acts on the system from $k = 500$. The simulation results for Example C are presented in Fig. \ref{fig_3}.%The simulation trajectories of the regulated state $x^1$ and the state $x^2$ are shown in Figs. \ref{fig_3}(a)-(b). The trajectory of the disturbance is shown in Fig. \ref{fig_5}.
\begin{figure}[htbp]
\begin{center}
%        \centering
        \includegraphics[width=9.0cm,height=7cm]{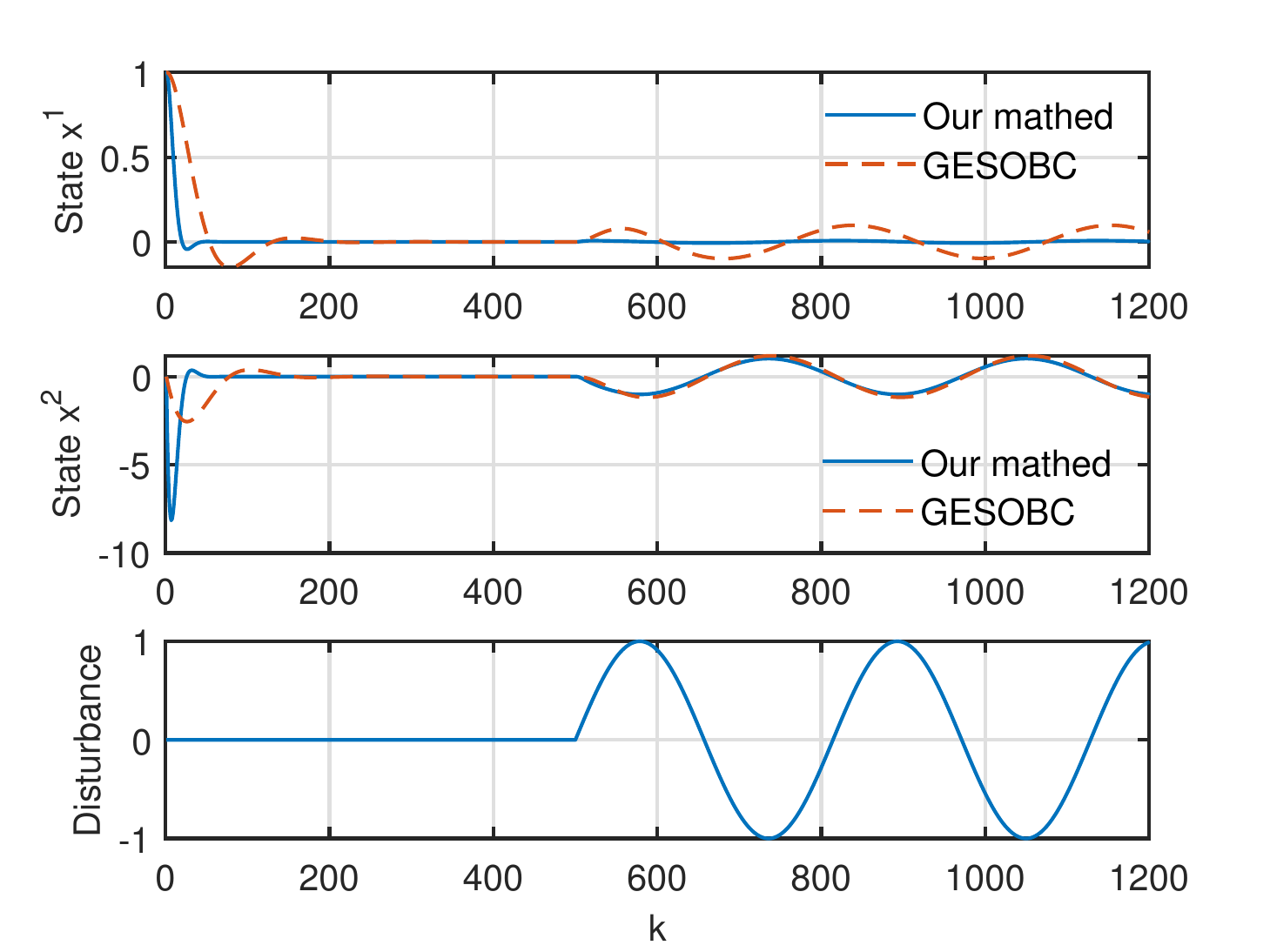}
%        \centerline{(b) State $x^2$}
    \end{center}
    \caption{Simulation result of Example C}\label{fig_3}
\end{figure}

As shown in Fig. \ref{fig_3}, our method is more effective than GESOBC for disturbance rejection because the proposed method almost completely removes the disturbance in the regulated state $x^1$. The results of this example demonstrate that our method can handle time-varying disturbances and the effect of disturbance rejection is clear.

\subsection{Example D: Application to an aero-engine system}\label{sec5}
In an aero-engine nozzle performance test, it is necessary to test the effects of changes in nozzle area on the engine performance parameters. It is typically required that the engine rotor speed and other states be stabilized rapidly. In this scenario, the change in nozzle area for the mass fuel flow control loop can be considered as a mismatched disturbance. The existing PID control method is mostly used for aero-engines, and it is difficult to balance disturbance rejection with optimal performance. Therefore, the proposed method was applied to the control of an aero-engine and the results were compared to the effect of PID control to demonstrate the effectiveness of the proposed method. For considering the effects of the nozzle area change of a mixed-exhaust turbofan engine system \citep{peng2016} on engine performance to verify the effectiveness of the proposed disturbance rejection method, the target system is defined as follows:
\begin{align}{}\label{f5.1}
\dot{x}(t)=Ax(t)+{B}{u}(t)+E{d}(t),
%{y(t)=Cx(t)},\\
%{({y_o})(t)=c_ox(t)},\notag
\end{align}

where $x(t)=[ \Delta \bar n_l\quad \Delta \bar n_h]'$ represents the state variables, ${y}(t)=[ \Delta \bar n_l\quad \Delta \bar n_h]'$ is the control output, ${u}(t)= \Delta w_f$ is the control input, ${d}(t)=\Delta A_8$ represents the disturbance caused by $A_8$ area changes, and $({y_o})(t)=\Delta \bar n_h$ is the regulated output. $\Delta \bar n_{l} =\bar n_{l}-\bar n_{l,0}$, $\Delta \bar n_{h} = \bar n_{h}-\bar n_{h,0}$, $\Delta \bar w_f=\bar w_f-\bar w_{f,0}$, $\Delta \bar A_8= \bar A_8-\bar A_{8,0}$. $n_l$ represents the low-pressure rotor speed (\%), $n_h$ is the high-pressure rotor speed (\%), $A_8$ represents the nozzle throat area ($m^2$), and $w_f$ is the mass fuel flow ($kg/s$). $\bar n_{l}$, $\bar n_{h}$, $\bar w_{f}$, and $\bar A_{8}$ respectively represent the values of $n_{l}$, $n_{h}$, $w_{f}$, and $A_{8}$ normalized by their maximum values. $\bar n_{l,0}$, $\bar n_{h,0}$, $\bar w_{f,0}$, and $\bar A_{8,0}$ are steady-state point values.

The working point is defined by the altitude $H=0km$ and airplane Mach number $Ma=0$, and the stable point state parameters are $n_{h,0}=11098.02rpm$, $n_{l,0}=9495.945rpm$, $w_{f,0}=0.6229kg/s$, and $A_{8.0}=A_{8min}=0.27839m^2$ with the coefficient matrix
\begin{align}\label{f5.2}
&{A}=\begin{bmatrix}-1.76&-1.34\cr2.70&-7.21\end{bmatrix},B=\begin{bmatrix}0.57\cr 0.82\end{bmatrix},\notag\\
&E=\begin{bmatrix}0.98& 2.26\end{bmatrix}',{c}_o=\begin{bmatrix}0&1\end{bmatrix}.
\end{align}
\begin{remark}
In our experiment, when the $A8$ area changed at a certain rate, the speed was more effectively stabilized by the proposed method compared to the reference. The change in the area of $A_8$ can be considered as a disturbance $d(t)$. The command to adjust the area of $A_8$ is issued by the controller and the disturbance can be considered as a known quantity.
\end{remark}

From (\ref{f5.1}) and (\ref{f5.2}), one can see that the coefficient ratios of the disturbance and control inputs into different channels of the system are different with $rank(B,E)>rank(E)$. In other words, the disturbances in the system are mismatched. The control objective is to ensure that $\bar n_h$ is stable at 77\% when the area of $A_8$ increases from $A_{8 min}$ to $A_{8 ref}$ at the maximum rate. Specifically, the area of $A_8$ within $0.5s$ increases from $0.27839m^2$ to $0.342039m^2$ with a maximum rate of $0.1273m^2/s$. The profile of $A_8$ (disturbance) is presented in Fig. \ref{fig_4}. 
To facilitate control of the digital system, the system is discretized with a sample interval $T_s=0.02s$. In the disturbance rejection control method proposed in Theorem \ref{t1}, the terminal time $T_{N}$ is set to $1s>0.5s$, so $T_{N}$ can be divided into 75 steps according to $T_s$ and the terminal condition $P_{N+1}=O$ (zero matrix). $P_k$, $f_k$, $h_k$, $M_k$, and $\Upsilon_{k}$ can be calculated using (\ref{f2.4}) to obtain $u_{k}$. The proposed control method was compared to PID control to track the reference and disturbance rejection effects, and the parameters in the PID method were set to $K_p=20$, $K_i=600$, and $K_d=0.1$ through optimal tuning.
%\begin{figure}[htbp]
%    \begin{center}      
%        \includegraphics[width=8.8cm,height=5.5cm]{EE2.eps}
%        \caption{Area change of $A_8$ (disturbance).}\label{fig_6}
%    \end{center}%
%\end{figure}

The response curves of the aero-engine system obtained using the proposed method and PID are presented in Fig. \ref{fig_4}. %The control input fuel flow $w_f$ variation is illustrated in Fig. \ref{fig_8}. 
Fig. \ref{fig_4} reveals that the proposed control method achieves a fast and smooth transition from the set point to a relatively high-pressure rotor speed $\bar n_h$ in front of the disturbance variation. Therefore, it can be concluded that the proposed method is superior to PID control for balancing disturbance rejection control and output optimization, and that excellent disturbance rejection performance can reduce the impact on the engine. 

These results demonstrate that the proposed method achieves satisfactory performance in terms of suppressing mismatched disturbances. When the controller obtains information about a disturbance, the influence of the disturbance can be quickly and completely eliminated. In this case, the proposed method is more effective than the PID algorithm.

%\begin{figure}[htbp]
%    \begin{center}      
%        \includegraphics[width=8.8cm,height=5.5cm]{EE1.eps}
%        \caption{Response curves of the relative speed of $\bar n_h$.}\label{fig_7}
%    \end{center}
%\end{figure}
\begin{figure}[htbp]
    \begin{center}      
        \includegraphics[width=9.0cm,height=7cm]{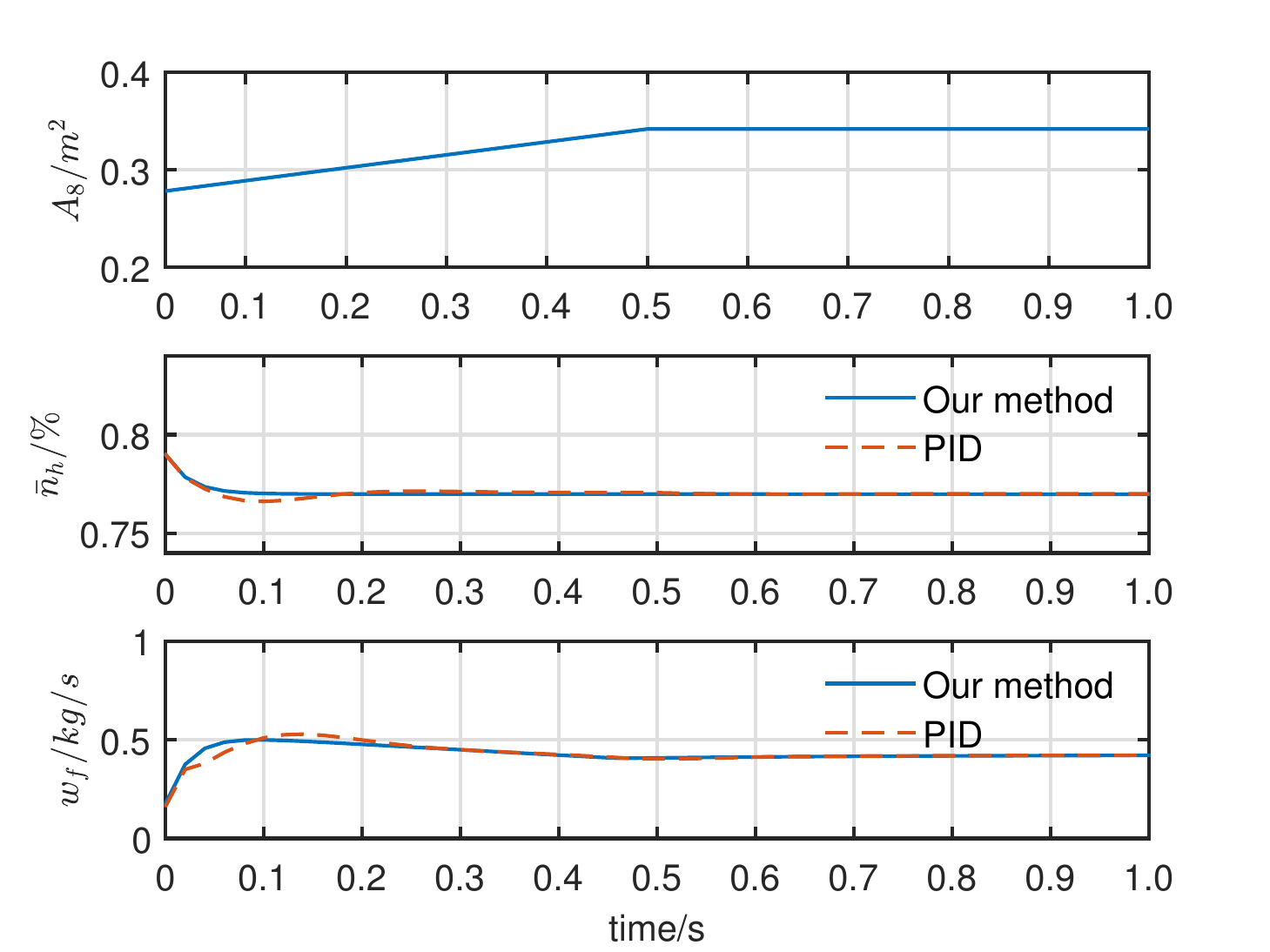}
        \caption{Simulation result of Example D}\label{fig_4}
    \end{center}
\end{figure}

%\begin{figure}[htbp]
%    \begin{minipage}[t]{0.5\linewidth}
%        \centering
%        \includegraphics[width=\textwidth]{EE1.eps}
%        \caption{Response curves of the relative speed of $\bar n_h$.}\label{fig_7}
%    \end{minipage}%
%    \begin{minipage}[t]{0.5\linewidth}
%        \centering
%        \includegraphics[width=\textwidth]{EE3.eps}
%        \caption{Change curve of control input fuel flow $w_f$.}\label{fig_8}
%    \end{minipage}
%\end{figure}

\section{Conclusion}\label{sec6}
The mismatched disturbance rejection problem was transformed into an LQT problem by introducing a new quadratic performance index that considers the regulated state to track a reference and minimize the effects of disturbances. The necessary and sufficient conditions for the solvability of the problem and sufficient conditions for the disturbance rejection controller and stability of the system were given in the finite and infinite horizons, respectively. It is noteworthy that this approach weakens the assumption of controllability. Several examples were presented to illustrate the effectiveness of the proposed method. Although our proposed controller exhibited excellent performance, it requires known disturbances. In the future, we will further study the design of a mismatched unknown disturbance (uncertainty) rejection controller based on the concepts and methods presented in this paper.

\end{document}